\documentclass[a4paper,12pt]{article}
\usepackage{mathtools}
\usepackage{amsmath}
\usepackage{amssymb}
\usepackage{amsthm}
\usepackage{wasysym}
\usepackage{tikz}
\usepackage{tikz-network}
\usepackage{graphicx}
\usepackage{eucal}
\usepackage{comment}
\usepackage{float}
\usepackage{caption}
\usepackage{mathrsfs}
\usepackage[left=2cm,right=2cm]{geometry}
\usepackage{bbm}
\usepackage{enumerate}
\usepackage{cite}
\usepackage{hyperref}

\theoremstyle{plain}%default
\newtheorem{thm}{Theorem}
\newtheorem{lem}[thm]{Lemma}
\newtheorem{cor}[thm]{Corollary}

\theoremstyle{definition}
\newtheorem{defn}[thm]{Definition}
\newtheorem{example}{Example}

\theoremstyle{definition}

\theoremstyle{definition}
\newtheorem{rem}{Remark}

\theoremstyle{definition}

\def\B{\mathscr{B}}

\def\M{\mathcal{M}}

\begin{document}
	\title{Unimodular Bicyclic Graphs}
	\author{
		Md Isheteyak Zaffer\footnote{Department of Mathematical Sciences, Tezpur University, Tezpur, Assam-784028, India; email: isheteyak.zaffer@gmail.com.}
		}
	\date{}
	
	\maketitle
	
	\begin{abstract}
		Let $G$ be a simple undirected graph with adjacency matrix $A(G)$. A graph $G$ is said to be \emph{unimodular} if $\det A(G)\in\{-1,1\}$. A connected graph with $m$ vertices and $m+k-1$ edges is called \emph{$k$-cyclic}; in particular, a bicyclic graph has $m$ vertices and $m+1$ edges. Unimodular unicyclic graphs have been completely characterized. In this paper, we investigate the corresponding problem for bicyclic graphs. We provide a complete characterization of unimodular bicyclic graphs and determine all possible values of $\det A(G)$ for a bicyclic graph $G$. Our study is motivated by the central role of unimodular graphs in the theory of graph inverses and their connections with eigenvalue reciprocity and other spectral properties of graphs.
	\end{abstract}
	
	
	
	\section{Introduction}
	All graphs in this article are undirected, simple, and connected. $E(G)$ and $V(G)$ denotes the edge set and the vertex set of a graph $G$, respectively. If a connected graph $G$ has $|V(G)|=m$ and $|E(G)|=m + 1$, then it is \textbf{bicyclic}. It follows that a bicyclic graph contains at least two cycles. If the cycles in a bicyclic graph $B$ are edge-disjoint, then $B$ is called an $\infty-$type bicyclic graph; otherwise, it is a $\theta-$type bicyclic graph. It is easy to see that the cycles in an $\infty-$type bicyclic graph share at most one vertex, while any two cycles in a $\theta-$type bicyclic graph have at least two vertices in common. The set of bicyclic graphs is denoted by $\B$. Additionally, $\B(\theta)$ and $\B(\infty)$ denotes the subset of $\theta-$type and $\infty-$type bicyclic graphs. We use $[u,v]$ or $u\sim v$ to senote an edge between vertices $u$ and $v$. The \textit{adjacency matrix} $A(G)=[a_{ij}]$ of a graph $G$ is an $n\times n$ matrix such that $$a_{ij}=\left\{\begin{array}{ll}
		1 & \text{ if } i\sim j,\\
		0 & \text{ otherwise.}
	\end{array}\right.
	$$
	
	A subgraph of $G$ whose components are exclusively $K_2$ graphs or cycles is termed as \textit{elementary subgraph}. When such a subgraph $H$ satisfies $V(H) = V(G)$, it is referred to as a \textit{spanning elementary subgraph}. If each component of a spanning subgraph is $K_2$, then the subgraph is referred as a \textit{perfect matching}. All perfect matchings is a spanning elementary subgraphs, and every graph admitting a perfect matching has an even number of vertices.
	
	The following lemma, due to Harary \cite{FH}, plays a central role in this work, as it offers a formula for determining the determinant of a graph.
	\begin{lem}(\cite{FH}, Theorem 3)\label{det}
		Suppose that $G$ is a simple graph with the adjacency matrix $A(G)$. Then $$\det(A(G))= \sum_{H}^{}2^{|C_H|}(-1)^{|C_H|+|E(H)|},$$
		where the sum is done over all spanning elementary subgraphs $H$ of $G$ and $|C_H|$ is the number of cycle components in $H$.
	\end{lem}

	A graph $G$ is termed \textit{unimodular} if the determinant of $A(G)$ equals either $1$ or $-1$. Unimodularity plays a significant role in understanding graph invertibility and certain spectral properties. For instance, if $G$ is \textit{sign-invertible}, it must also be unimodular, see Buckley et al. \cite{BH} for the definition of sign-invertibility. We say that $G$ possesses \textit{strong reciprocal eigenvalue property} (SR-property) when $\lambda$ and $\frac{1}{\lambda}$ occur as an eigenvalue of $A(G)$ with the same multiplicity. It is worth noting that a graph with SR-property is always unimodular. Graphs exhibiting this eigenvalue properties have been extensively investigated in the literature (see \cite{MN} and references therein).
	
	It is a well-established fact that a tree is unimodular precisely when it possesses a perfect matching. Akbari and Kirkland \cite{AK} have characterized unimodular unicyclic graphs, while Basumatary and Sarma \cite{KS} have given conditions under which a graph that has exactly one perfect matching is unimodular. Recently, Jaume et al. \cite{DDC} provided another characterization of graphs that possesses exactly one perfect matching that are unimodular, expressed in terms of the barbell-part. This leads naturally to the question: \textit{Can one characterize unimodular bicyclic graphs in general?} The present work addresses this question and provides a complete answer. In \cite{AK}, it was shown that possessing a unique perfect matching is sufficient as well as necessary for a unicyclic graph to be unimodular. However, \textit{does this equivalence extend to bicyclic graphs?} Example \ref{ex1} indicates that the answer is \textit{no}. We show that, the existence of a unique perfect matching is neither sufficient nor necessary for a bicyclic graph to be unimodular. Specifically, we show that for $B \in \B(\infty)$, possessing a unique perfect matching is necessary but not sufficient for $B$ to be unimodular, whereas for $B \in \B(\theta)$, it is sufficient but not necessary. In fact, this article provides a complete identification of bicyclic graphs that are unimodular, naturally extending the work in \cite{KS}.
	
	\begin{example}\label{ex1}
		The graph $B_1$ has exactly one perfect matching, yet $\det A(B_1) = 3$, while $B_2$ has three perfect matchings despite having $\det A(B_2) = -1$.
		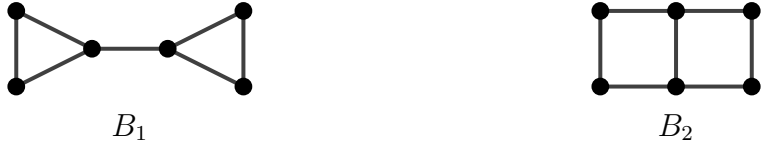
\begin{figure}[H]
			\begin{center}
				\begin{tabular*}{0.6\linewidth}{@{\extracolsep{\fill}} cc}
					\begin{tikzpicture}
							\SetVertexStyle[MinSize=0.2,FillColor=black]
							\Vertex[x=0.5]{1}
							\Vertex[x=1.5,y=0.5]{2}
							\Vertex[x=1.5,y=-0.5]{3}
							\Vertex[x=-1.5,y=0.5]{2'}
							\Vertex[x=-1.5,y=-0.5]{3'}
							\Vertex[x=-0.5]{1'}
							
							\Edge(1)(1')
							\Edge(1)(2)
							\Edge(1)(3)
							\Edge(2)(3)
							\Edge(1')(2')
							\Edge(1')(3')
							\Edge(2')(3')
					\end{tikzpicture} & \begin{tikzpicture}
							\SetVertexStyle[MinSize=0.2,FillColor=black]
							\Vertex[]{1}
							\Vertex[x=1]{2}
							\Vertex[x=2]{3}
							\Vertex[x=0,y=1]{6}
							\Vertex[x=1,y=1]{5}
							\Vertex[x=2,y=1]{4}
							
							\Edge(1)(2)
							\Edge(2)(3)
							\Edge(3)(4)
							\Edge(4)(5)
							\Edge(5)(6)
							\Edge(6)(1)
							\Edge(5)(2)
				\end{tikzpicture}\\
				$B_1$ & $B_2$\\
				\end{tabular*}
			\end{center}\caption{Bicyclic graphs $B_1$ and $B_2$ in Example \ref{ex1}.}\label{fig1}
		\end{figure}
	\end{example}
	
	The structure of this article is as follows. Preliminary results required for the subsequent sections are presented in Section 2. In this section, it is also shown that a graph without a perfect matching cannot be unimodular. Section 3 focuses on investigating unimodular graphs in $\B(\infty)$ that possess a perfect matching, while Section 4 extends this analysis to graphs in $\B(\theta)$. Finally, in Section 5, we determine the set of possible values for $\det A(B)$ when $B \in \B$.
	
	\section{Some general results}
	
	In this section, we present some preliminary results on graphs in $\B$ and their spanning elementary subgraphs. Let $K$ be a subgraph of a graph $B$; then $B - K$ denotes the induced subgraph of $B$ on the set $V(B) - V(K)$. Let $\M$ represent a perfect matching in $B$. Throughout the article, we take $|V(B)| = n$, which implies that $|E(\M)| = \frac{n}{2}$. We use $|\Gamma|$ and $|P|$ to denote the length of a cycle $\Gamma$ and a path $P$, respectively. We denote a path from $i$ to $j$ by $i\leadsto j$ or $[i,\dots,j]$.
	
	\begin{rem}\label{det2}
		Let $B\in\B$ be a graph. Assume that $B$ has $m_0$ perfect matchings, $m_{1i}$ spanning elementary subgraphs containing $\Gamma_i$, and $m_2$ spanning elementary subgraphs containing cycles $\Gamma$ and $\Gamma'$. By Lemma \ref{det},
		\begin{equation*}
			\begin{split}
				\det A(B)=& m_0\times 2^0(-1)^{\frac{n}{2}}+\sum_{i}m_{1i}\times 2^1(-1)^{1+\frac{n-|\Gamma_i|}{2}+|\Gamma_i|}+m_2\times2^2(-1)^{2+\frac{n-|\Gamma|-|\Gamma'|}{2}+|\Gamma|+|\Gamma'|}\\
						=& m_0\times(-1)^\frac{n}{2}-2\sum_{i}m_{1i}\times(-1)^{\frac{n}{2}+\frac{|\Gamma_i|}{2}}+4m_2\times(-1)^{\frac{n}{2}+\frac{|\Gamma|+|\Gamma'|}{2}}\\
						=&(-1)^\frac{n}{2}\left[m_0-2\sum_{i}m_{1i}(-1)^{\frac{|\Gamma_i|}{2}}+4m_2(-1)^{\frac{|\Gamma|+|\Gamma'|}{2}}\right]\\
			\end{split}
		\end{equation*}
	\end{rem}

	We have the following lemma for a graph without perfect matching.
	
	\begin{lem}\label{nonperf}
		Let $B\in\B$ be such that $B$ does not possess any perfect matching. Then $B$ is not unimodular.
	\end{lem}
	\begin{proof}
		Suppose that $B$ lacks a spanning elementary subgraph, then by Lemma \ref{det}, we have $\det A(B) = 0$. Hence, we may assume that $B$ admits a spanning elementary subgraph. Since $B$ has no perfect matching, every such subgraph must contain at least one cycle. From Remark \ref{det2}, it follows that $\det A(B) \equiv 0 \pmod 2$ (because $m_0 = 0$). Consequently, $B$ is not unimodular.
	\end{proof}

	By Lemma~\ref{nonperf}, every unimodular graph $B\in\B$ must possess a perfect matching. Before proceeding further, we present the following lemma concerning graphs with more than one perfect matching. Although the result is straightforward and appears as Proposition~6 in \cite{I}, we include a proof for the sake of completeness.

	\begin{lem}\label{mperf}
		Let $G$ be a graph. Then following are equivalent.
		\begin{enumerate}[1]
			\item $G$ contains an even cycle $\Gamma$ for which $G - \Gamma$ possesses a perfect matching.
			\item $G$ has more than one perfect matching.
		\end{enumerate}
	\end{lem}
	\begin{proof}
		Suppose that $G$ contains an even cycle $\Gamma = [u_1, \dots, u_{n}]$ such that $G - \Gamma$ admits a perfect matching $\M_0$. Then the sets $\M_0 \cup \{[u_1, u_2], \dots, [u_{n-1}, u_n]\}$ and $\M_0 \cup \{[u_1, u_n], \dots, [u_3, u_2]\}$ form two distinct perfect matchings of $G$. 
		
		Conversely, assume that $G$ possesses two distinct perfect matchings $\M$ and $\M'$. We begin by showing that $G$ contains an even cycle. Observe that the symmetric difference $\M\Delta\M'$ is non-empty. If $[u,v]\in \M\Delta\M'$, then $deg(v)\geq2$ and $deg(u)\geq 2$, if not, say $deg(u)=1$, then $[u,v]\in \M\Delta\M'$ since $\M$ and $\M'$ are both perfect matchings. Without loss of generality, assume that $[u,v]\in\M\setminus\M'$. Then, there exists vertices $u_1$ and $v_1$ such that $[u,u_1],[v,v_1]\in \M'\setminus\M$. 
		
		If $u_1$ and $v_1$ are adjacent, then $\Gamma=[u,u_1,v_1,v]$ is an even cycle in $G$. Otherwise, we find vertices $u_2$ and $v_2$ such that $[u_1,u_2],[v_1,v_2]\in\M\setminus\M'$. This process continues till we find vertices $u_j$ and $v_j$ such that $u_j\sim v_j$, which necessarily happens since $G$ is a finite graph. Clearly, $\Gamma=[u,u_1\dots,u_j,v_j,\dots,v_1,v]$ is the desired even cycle in $G$. Finally, observe that $\M\setminus\{[u,v],[u_1,u_2],\dots,[v_2,v_1]\}$ is a perfect matching in $G-\Gamma$.
	\end{proof}

	\begin{defn}\label{defn1}
		Let $G$ be a graph. A cycle $\Gamma$ in $G$ is called an \textit{independent cycle} if $G-\Gamma$ has a perfect matching.
	\end{defn}

	\begin{cor}\label{cor1}
		Let $G$ be a graph. If $G$ possess a unique perfect matching, then $G$ does not contain an independent cycle.
	\end{cor}
	\begin{proof}
		Suppose that $G$ contains an independent cycle $\Gamma$. Since $G$ has exactly one perfect matching, Lemma \ref{mperf} implies that $\Gamma$ is odd. $|V(G-\Gamma)|$ is even as $G-\Gamma$ possesses a perfect matching. Observe that $|V(G - \Gamma)| = |V(G)| - |\Gamma|$ which implies that $|V(G)|$ is odd. This contradicts the fact that there is a perfect matching in $G$. Hence, $G$ does not contain an independent cycle.
	\end{proof}

	\begin{rem}\label{rem1}
		It follows from Lemma \ref{mperf}, that an independent cycle in a graph with more than one perfect matching is even, whereas in a graph with no perfect matching it is odd.
	\end{rem}

	%Since $\infty-$type bicyclic graphs and $\theta-$type bicyclic graphs has structural differences, we separate their study into different sections.
	
	\section{Unimodularity in $\B$}
	
	We first need to determine the number of spanning elementary subgraphs of $B\in\B$ when $B$ has a perfect matching. These are already done in \cite{I}.
	
	\begin{lem}[\cite{I}, Lemma~12]\label{1ind}
		Let $B \in \B$ be a graph that possesses a perfect matching. If $B$ has a unique independent cycle $\Gamma$, then $B$ has three spanning elementary subgraphs: two perfect matchings and a spanning elementary
		subgraph containing $\Gamma$.
	\end{lem}

	As an immediate consequence we have the following lemma.
	\begin{lem}\label{notuni}
		Let $B\in\B$ be a graph with more than one perfect matchings. If $B$ has exactly one independent cycle, then $B$ is not unimodular.
	\end{lem}
	\begin{proof}
		Let $\Gamma$ be the independent cycle in $B$. Then, by Lemma \ref{1ind}, $B$ has two perfect matchings and a spanning elementary subgraph containing $\Gamma$. It follows from Remark \ref{det2} that
		\[
		\begin{split}
			\det A(B)&=(-1)^{\frac{n}{2}}\left[2-2(-1)^{\frac{|\Gamma|}{2}}\right]\\
			&=2(-1)^{\frac{n}{2}}\left[1-(-1)^{\frac{|\Gamma|}{2}}\right]\\
			&\in\{0, \pm 4\}.
		\end{split}
		\]
	\end{proof}

	In the following sections, we examine the unimodularity of graphs in $\B$ that possess a perfect matching. Given the structural distinctions between $\infty$-type and $\theta$-type graphs, we conduct separate analyses for each type.
	
	\subsection{Unimodularity in $\B(\infty)$}

	\begin{lem}[\cite{I}, Lemma~16]\label{2ind}
		Let $B \in \B(\infty)$ be a graph that possesses a perfect matching. Let $\Gamma$ and $\Gamma'$ be cycles in $B$. If both $\Gamma$ and $\Gamma'$ are independent in $B$, then $B$ has nine spanning elementary subgraphs: four
		perfect matchings, two spanning elementary subgraphs containing only $\Gamma$ as cycle component,
		two containing only $\Gamma'$, and one containing both $\Gamma$ and $\Gamma'$.
	\end{lem}

	\begin{lem}\label{cor2}
		Let $B\in\B(\infty)$ with more than one perfect matching. Then $\det A(B)\in \{0,\pm 4,\pm 16\}$ and $B$ is not unimodular.
	\end{lem}
	\begin{proof}
		Since $B$ has more than one perfect matching, it contains at least one independent cycle. Suppose first that $B$ has exactly one independent cycle $\Gamma$. Then, by Lemma~\ref{notuni}, $\det A(B)\in\{0, \pm 4\}$.
		
		Now suppose that both $\Gamma$ and $\Gamma'$ are independent cycles in $B$. Then, by Lemma~\ref{2ind} and Remark~\ref{det2},
		\[
		\begin{split}
			\det A(B)
			&=(-1)^{\frac{n}{2}}
			\left[
			4
			-4(-1)^{\frac{|\Gamma|}{2}}
			-4(-1)^{\frac{|\Gamma'|}{2}}
			+4(-1)^{\frac{|\Gamma|+|\Gamma'|}{2}}
			\right]\\
			&\in\{0,\pm 16\}.
		\end{split}
		\]
	\end{proof}

	%Observe that $B$ has exactly three spanning elementary subgraphs, the two perfect matchings and $\Gamma_1\cup \M_0$. The number edges in each perfect matchings in $B$ equals $\frac{n}{2}$, whereas $|\M_0|=\frac{n-|\Gamma|}{2}$. Hence, $\det A(B)=2^0(-1)^{\frac{n}{2}}+2^0(-1)^{\frac{n}{2}}+2^1(-1)^{1+m+\frac{n-|\Gamma|}{2}}=2\times (-1)^{\frac{n}{2}}\left[1-(-1)^\frac{|\Gamma_1|}{2}\right]\in \{0,\pm 4\}$. it is not difficult to see that $B$ has four perfect matchings, two spanning elementary subgraphs containing $\Gamma_1$, two containing $\Gamma_2$, and one containing both $\Gamma_1$ and $\Gamma_2$. Hence, $\det A(B)=4\times (-1)^\frac{n}{2}+4\times (-1)^{1+\frac{|\Gamma_1|}{2}}+4\times(-1)^{1+\frac{|\Gamma_2|}{2}}+4\times(-1)^{2+\frac{|\Gamma_1|+|\Gamma_2|}{2}}\in\{0,\pm 16\}$.

	It is natural to ask \textit{Does there even exists a unimodular $\infty-$type bicyclic graph?} Example \ref{iuni} shows that there is such a graph. Notice that $B$ possess exactly one perfect matching.
	
	\begin{example}\label{iuni}
		An example of $\infty-$type unimodular graph.
		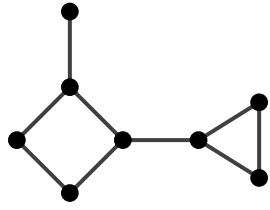
\begin{figure}[H]
			\begin{center}
				\begin{tikzpicture}
					\SetVertexStyle[MinSize=0.1,FillColor=black]
					\Vertex[]{1}
					\Vertex[x=-1.4]{2}
					\Vertex[x=-0.7,y=0.7]{3}
					\Vertex[x=-0.7,y=-0.7]{4}
					\Vertex[x=1]{5}
					\Vertex[x=1.8,y=0.5]{6}
					\Vertex[x=1.8,y=-0.5]{7}
					\Vertex[x=-0.7,y=1.7]{8}
					
					\Edge(1)(3)
					\Edge(1)(4)
					\Edge(4)(2)
					\Edge(3)(2)
					\Edge(3)(8)
					\Edge(1)(5)
					\Edge(5)(6)
					\Edge(6)(7)
					\Edge(7)(5)
				\end{tikzpicture}
			\end{center}\caption{A unimodular graph in $\B(\infty)$.}\label{uniinf}
		\end{figure}
	\end{example}

	The following lemma is obvious.
	
	\begin{lem}\label{iuniqueperuni}
		If $B\in\B(\infty)$ is unimodular, then $B$ has exactly one perfect matching.
	\end{lem}
	\begin{proof}
		It follows from Lemma \ref{nonperf} that $B$ has at least one perfect matching and from Lemma \ref{cor2} $B$ has exactly one perfect matching.
	\end{proof}

	We naturally wonder whether every $B\in\B(\infty)$ having a unique perfect matching is unimodular. Example \ref{ex1} already says that the answer is, in fact, \textit{negative}. Consequently, when a graph in $\B(\infty)$ has a unique perfect matching, it is essential to determine the circumstances in which it is unimodular.
	
	\begin{lem}\label{iuniq}
		Let $B\in\B(\infty)$ be a graph with exactly one perfect matching, and let $\Gamma$ and $\Gamma'$ be cycles in $B$. If $B-(\Gamma\cup\Gamma')$ has a perfect matching or is empty, then $\det A(B)\in \{\pm3,\pm5\}$.
	\end{lem}
	\begin{proof}
		Let's denote the perfect matching in $B$ by $\M$, and let $\M_1$ be the perfect matching in $B - (\Gamma \cup \Gamma')$ (where $\M_1$ is empty if $B - (\Gamma \cup \Gamma')$ is empty). Both $|V(B)|$ and $|V(B - (\Gamma \cup \Gamma'))|$ are even, since each admits a perfect matching. This means that $\Gamma$ and $\Gamma'$ are vertex-disjoint; if they were not, then one of them—say $\Gamma$—would be odd, as $|V(B - (\Gamma \cup \Gamma'))| = |V(B)| - |V(\Gamma \cup \Gamma')| = |V(B)| - (|\Gamma| + |\Gamma'| - 1)$. Consequently, there exists vertices $j \in V(\Gamma)$ and $j' \notin V(\Gamma)$ for which $[j, j'] \in \M$. Since $j' \in V(\M_1)$, we find a vertex $j_1$ such that $[j', j_1] \in \M_1$. By continuing this process alternately for $\M$ and $\M'$, we would obtain an infinite path $[j, j', j_1, j_1', \dots]$ in $B$, which is impossible.
		
		Thus, $\Gamma$ and $\Gamma'$ are vertex-disjoint. Hence, we find another spanning elementary subgraph $H = \M_1 \cup \Gamma \cup \Gamma'$ of $B$, which means that $B$ possesses precisely two spanning elementary subgraphs. By Remark \ref{det2}, we have $\det A(B)=(-1)^\frac{n}{2}\left[1+4\times(-1)^{\frac{|\Gamma|+|\Gamma'|}{2}}\right]\in \{\pm3,\pm5\}$.
	\end{proof}

	The following theorem gives complete characterization of unimodular graphs in $B\in\B(\infty)$.
	
	\begin{thm}\label{thm1}
		Let $B\in\B(\infty)$ with cycles $\Gamma$ and $\Gamma'$. Then $B$ is unimodular if and only if $B$ satisfies following conditions.
		\begin{enumerate}
			\item $B$ possess a unique perfect matching, and
			\item $B-(\Gamma\cup\Gamma')$ is neither empty nor possesses a perfect matching.
		\end{enumerate}
	\end{thm}
	\begin{proof}
		If $B$ satisfies the given conditions, then the unique perfect matching is the only spanning elementary subgraph in $B$, as neither $\Gamma$ nor $\Gamma'$ is independent by Lemma \ref{mperf}. Consequently, $\det A(B)=(-1)^\frac{n}{2}=\pm1$.
		
		The converse follows from Lemma \ref{iuniqueperuni} and Lemma \ref{iuniq}.
	\end{proof}

	\begin{cor}\label{cor3}
		Let $B\in\B(\infty)$ be a graph with a unique perfect matching. If two cycles in $B$ have a vertex in common, then $B$ is unimodular. 
	\end{cor}
	\begin{proof}
		From the proof of Lemma \ref{iuniq}, it follows that if $\Gamma$ and $\Gamma'$ share a vertex then there is no perfect matching in $B - (\Gamma \cup \Gamma')$.
	\end{proof}

	\begin{cor}\label{cor4}
		Let $B\in\B(\infty)$ be a graph with a unique perfect matching. If $B$ has an even cycle, then $B$ is unimodular.
	\end{cor}
	\begin{proof}
		Let $\Gamma$ and $\Gamma'$ be the cycles in $B$, with $\Gamma$ assumed to be an even cycle. By Corollary \ref{cor1}, we can assume that $V(\Gamma)\cap V(\Gamma')=\emptyset$. Observe that if $B - V(\Gamma \cup \Gamma')$ has a perfect matching, then $\Gamma'$ is an even cycle, since the orders of $B$, $B - V(\Gamma \cup \Gamma')$, and $\Gamma$ are all even. Consequently, a perfect matching in $B - V(\Gamma \cup \Gamma')$, combined with a perfect matching in $\Gamma$ and two perfect matchings in $\Gamma'$, would yield two distinct perfect matchings in $B$, contradicting the assumption that $B$ has exactly one perfect matching. Therefore, $B - V(\Gamma \cup \Gamma')$ does not possess a perfect matching. By Theorem \ref{thm1}, the conclusion holds.
	\end{proof}

%	\begin{proof}
%		Let $\Gamma$ and $\Gamma'$ be a cycle in $B$. Observe that $V(\Gamma)\cap V(\Gamma')$ is empty, otherwise any spanning elementary subgraph can have at most one cycle and since a graph with a unique perfect matching does not contain any independent cycle, $B$ has exactly one spanning elementary subgraph $\M$. This implies that $\det A(B)=\pm1$, which contradicts that $B$ is not unimodular. Now, assume that $\Gamma$ is an even cycle and $\Gamma'$ is odd, then $B-V(\Gamma\cup\Gamma')$ has odd number of vertices and therefore does not possess any perfect matching, contradicting Theorem \ref{thm1}. Therefore $\Gamma$ and $\Gamma'$ both are even. If $B-V(\Gamma\cup\Gamma')$ has a perfect matching $\M_0$, then $\M_0$ along with two perfect matchings from $\Gamma$ and $\Gamma'$ is four perfect matchings in $B$ contradicting that $B$ has unique perfect matching. Hence $\Gamma$ and $\Gamma'$ both are odd. 
%	\end{proof}

	\subsection{Unimodularity in $\B(\theta)$}
	
		In this subsection, we continue our investigation of unimodular graphs and identify unimodular $\theta-$type graphs. Contrary to $\infty-$type graphs, we have the following lemma for graphs $B\in\B(\theta)$.
		
		\begin{lem}\label{tuniquni}
			Let $B\in\B(\theta)$. If $B$ possesses a unique perfect matching, then it is unimodular.
		\end{lem}
		\begin{proof}
			Denote the unique perfect matching in $B$ by $\M$. Since any two cycles in $B$ share an edge, any spanning elementary subgraph of $B$ contains at most one cycle. By Corollary \ref{cor1}, $B$ does not contain any independent cycle and, consequently, has no spanning elementary subgraph that includes a cycle. Thus, $B$ has no spanning elementary subgraph other than $\M$. Therefore, from Remark \ref{det2}, we have $\det A(B) = (-1)^{\frac{n}{2}} = \pm 1$.
		\end{proof}
	
		Naturally, one might wonder whether the converse holds. However, as demonstrated by Example \ref{ex1}, the answer is clearly \textit{no}. This observation highlights the need for a deeper investigation into the unimodularity of $B\in\B(\theta)$ when $B$ possesses more than one perfect matching.
		
		\begin{rem}\label{paths}
			In a $\theta$-type bicyclic graph, there exist two vertices $i$ and $j$ along with three $i \leadsto j$ paths that are pairwise edge-disjoint. We denote these paths by $\mathcal{Q}_1$, $\mathcal{Q}_2$, and $\mathcal{Q}_3$. See Figure \ref{pdiag}. $B$ has three cycles given by $\Gamma_1=\mathcal{Q}_1\cup \mathcal{Q}_2$, $\Gamma_2=\mathcal{Q}_1\cup\mathcal{Q}_3$, and $\Gamma_3=\mathcal{Q}_2\cup\mathcal{Q}_3$.
		\end{rem}
		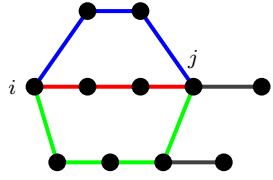
\begin{figure}[H]
			\begin{center}
				\begin{tikzpicture}
					\SetVertexStyle[MinSize=0.1,FillColor=black]
					\Vertex[label=$i$,position=left]{1}
					\Vertex[x=0.7]{2}
					\Vertex[x=1.4]{3}
					\Vertex[x=2.1,label=$j$,position=above]{4}
					\Vertex[x=0.3,y=-1]{5}
					\Vertex[x=1,y=-1]{6}
					\Vertex[x=1.7,y=-1]{7}
					\Vertex[x=0.7,y=1]{8}
					\Vertex[x=1.4,y=1]{9}
					\Vertex[x=3]{10}
					\Vertex[x=2.5,y=-1]{11}
					
					\Edge[color=red](1)(2)
					\Edge[color=red](2)(3)
					\Edge[color=red](3)(4)
					\Edge[color=green](1)(5)
					\Edge[color=green](5)(6)
					\Edge[color=green](6)(7)
					\Edge[color=green](7)(4)
					\Edge[color=blue](1)(8)
					\Edge[color=blue](8)(9)
					\Edge[color=blue](9)(4)
					\Edge(4)(10)
					\Edge(7)(11)
			\end{tikzpicture}
			\end{center}\caption{Paths described in Remark \ref{paths}, colored in red, blue and green.}\label{pdiag}
		\end{figure}
		
		By Remark \ref{paths}, it follows that if a path, say $\mathcal{Q}_1$, is even (resp. odd) and the other two paths are odd (resp. even), then $\Gamma_1$ and $\Gamma_2$ are odd, whereas $\Gamma_3$ is even. Similarly, if all the paths are even (or all are odd), then all cycles in $B$ are even. In particular, $B$ either contains exactly one even cycle or all its cycles are even.
		
		In view of Lemma~\ref{notuni}, it remains to consider the case where $B\in\B(\theta)$ has more than one independent cycle. The following lemma, proved in \cite{I}, describes all spanning elementary subgraphs of such graphs.

		\begin{lem}[\cite{I}, Lemma~21]\label{span}
			Let $B\in\B(\theta)$. If $B$ has more than one independent cycle, then $B$ has six spanning elementary subgraphs: three perfect matchings and one
			spanning elementary subgraph containing a cycle for each cycle in $B$.
		\end{lem}
%		\begin{proof}
%			Let $\Gamma_1$, $\Gamma_2$, and $\Gamma_3$ be the cycles in $B$. Since all the cycles in $B$ are independent, $B$ has a spanning elementary subgraph $H_i$ containing $\Gamma_i$ for each $i=1,2,3$. Furthermore, $B$ does not contain any other spanning elementary subgraph that contains a cycle. To count the number of perfect matchings in $B$, let $\mathcal{Q}_1 = [u, u_1, \dots, u_m, v]$, $\mathcal{Q}_2 = [u, v_1, \dots, v_n, v]$, and $\mathcal{Q}_3 = [u, w_1, \dots, w_p, v]$, where $m$, $n$, and $p$ are all even. Since all cycles are independent, $B - (\Gamma_1 \cup \Gamma_2 \cup \Gamma_3) = B - (\mathcal{Q}_1 \cup \mathcal{Q}_2 \cup \mathcal{Q}_3)$ has a perfect matching $\M$. It is clear that the following are three distinct perfect matchings in $B$:
%			
%			$\M \cup \{[u, u_1], \dots, [u_m, v]\} \cup \{[v_1, v_2], \dots, [v_{n-1}, v_n]\} \cup \{[w_1, w_2], \dots, [w_{p-1}, w_p]\}$,\\
%			$\M \cup \{[u_1, u_2], \dots, [u_{m-1}, u_m]\} \cup \{[u, v_1], \dots, [v_n, v]\} \cup \{[w_1, w_2], \dots, [w_{p-1}, w_p]\}$,\\
%			$\M \cup \{[u_1, u_2], \dots, [u_{m-1}, u_m]\} \cup \{[v_1, v_2], \dots, [v_{n-1}, v_n]\} \cup \{[u, w_1], \dots, [w_p, v]\}$.
%			Hence, $B$ has six spanning elementary subgraphs.
%		\end{proof}

	\begin{cor}\label{alodd}
		Let $B\in\B(\theta)$ be such that $B$ has a perfect matching. If $B$ has at least two independent cycles, then $|\mathcal{Q}_1|$, $|\mathcal{Q}_2|$, and $|\mathcal{Q}_3|$ are all odd.
	\end{cor}
	
	\begin{proof}
		By Lemma~\ref{span}, $B$ has a spanning elementary subgraph containing each cycle of $B$. Hence, all three cycles of $B$ are independent. Since $B$ has a perfect matching, each of these cycles is even. Consequently, the paths $\mathcal{Q}_1$, $\mathcal{Q}_2$, and $\mathcal{Q}_3$ are either all of even length or all of odd length.
		
		For $i=1,2$, let $\M_i$ be the unique perfect matching of $B-\Gamma_i$. Then,
		\[
		\M=\M_1\cap\M_2
		\]
		is a perfect matching of $B-(\Gamma_1\cup\Gamma_2)$. It follows that $|V(B-(\Gamma_1\cup\Gamma_2))|$ is even. Since $|V(B)|$ is also even, we conclude that $|V(\Gamma_1\cup\Gamma_2)|$ is even. Moreover,
		\[
		|V(\Gamma_1\cup\Gamma_2)|
		=
		|V(\Gamma_1)|
		+
		|V(\Gamma_2)|
		-
		|V(\Gamma_1\cap\Gamma_2)|.
		\]
		As $|V(\Gamma_1)|$ and $|V(\Gamma_2)|$ are even, it follows that
		\[
		|V(\Gamma_1\cap\Gamma_2)|
		=
		|V(\mathcal{Q}_1)|
		\]
		is even. Since the length of a path is one less than its number of vertices, $|\mathcal{Q}_1|$ is odd. Therefore, $|\mathcal{Q}_2|$ and $|\mathcal{Q}_3|$ are also odd.
	\end{proof}

		\begin{lem}\label{detAB}
			Let $B\in\B(\theta)$ be such that $B$ has a perfect matching. If $B$ has at least two independent cycles, then
			\[
			\det A(B)\in\{\pm1,\pm9\}.
			\]
		\end{lem}
		
		\begin{proof}
			By Lemma~\ref{span}, $B$ has exactly three perfect matchings and three spanning elementary subgraphs containing a cycle. Hence, by Remark~\ref{det2},
			\[
			\det A(B)
			=
			(-1)^{\frac{n}{2}}
			\left[
			3
			-
			2\left(
			(-1)^{\frac{|\Gamma_1|}{2}}
			+
			(-1)^{\frac{|\Gamma_2|}{2}}
			+
			(-1)^{\frac{|\Gamma_3|}{2}}
			\right)
			\right].
			\]
			
			By Corollary~\ref{alodd}, the paths $\mathcal{Q}_1$, $\mathcal{Q}_2$, and $\mathcal{Q}_3$ are all of odd length.
			
			Suppose first that
			\[
			|\mathcal{Q}_i|-|\mathcal{Q}_j|
			\equiv
			0
			\pmod4
			\quad\text{for all } i,j\in\{1,2,3\}.
			\]
			Then
			\[
			|\Gamma_i|
			=
			|\mathcal{Q}_j|+|\mathcal{Q}_k|
			\equiv
			2
			\pmod4,
			\]
			for every $i\in\{1,2,3\}$, where $\{i,j,k\}=\{1,2,3\}$. Consequently,
			\[
			(-1)^{\frac{|\Gamma_i|}{2}}=-1,
			\]
			for each $i$, and therefore
			\[
			\det A(B)
			=
			9(-1)^{\frac{n}{2}}
			\in
			\{\pm9\}.
			\]
			
			Now suppose that
			\[
			|\mathcal{Q}_i|-|\mathcal{Q}_j|
			\equiv
			2
			\pmod4
			\]
			for some distinct $i,j$. Then
			\begin{equation}\label{eq}
			|\Gamma_i|
			\equiv
			|\Gamma_j|
			\equiv
			2
			\pmod4,
			\qquad
			|\Gamma_k|
			\equiv
			0
			\pmod4,
			\end{equation}
			where $\{i,j,k\}=\{1,2,3\}$. Hence,
			\[
			(-1)^{\frac{|\Gamma_i|}{2}}
			=
			(-1)^{\frac{|\Gamma_j|}{2}}
			=
			-1,
			\qquad
			(-1)^{\frac{|\Gamma_k|}{2}}
			=
			1,
			\]
			which yields
			\[
			\det A(B)
			=
			(-1)^{\frac{n}{2}}
			\in
			\{\pm1\}.
			\]
			This completes the proof.
		\end{proof}
	
		\begin{cor}\label{cor19}
			Let $B\in\B(\theta)$ be such that $B$ has more than one independent cycle. Then $B$ is unimodular if and only if exactly one of the cycles $\Gamma_1$, $\Gamma_2$, and $\Gamma_3$ has length congruent to $0\pmod4$, while the other two have lengths congruent to $2\pmod4$.
		\end{cor}
		\begin{proof}
			Follows directly from Equation~\ref{eq}.
		\end{proof}
	
		\begin{cor}\label{cor5}
			Let $B\in\B(\theta)$ such that all cycles in $B$ are independent. Then $B$ is unimodular if and only if $|\mathcal{Q}_1|=2(q_1+q_2-q_3)+1$, $|\mathcal{Q}_2|=2(q_1-q_2+q_3)+1$, and $|\mathcal{Q}_3|=2(-q_1+q_2+q_3)-1$, where $q_1$, $q_2$, and $q_3$ are positive integers. 
		\end{cor}
		\begin{proof}
			Assume that $B$ has three independent cycles. From Corollary~\ref{cor19}, $\det A(B)=\pm1$ if and only if $|\Gamma_i|-|\Gamma_j|\equiv2\pmod4$ and $\Gamma_k=0\pmod4$ for distinct $i,j,k\in\{1,2,3\}$. Without loss of generality, assume that $|\Gamma_1|=4q_1+2, |\Gamma_2|=4q_2$, and $|\Gamma_3|=4q_3$, where $q_1,q_2$, and $q_3$ are positive integers. Therefore, $|\Gamma_1|=|\mathcal{Q}_1|+|\mathcal{Q}_2|, |\Gamma_2|=|\mathcal{Q}_1|+|\mathcal{Q}_3|$, and $|\Gamma_3|=|\mathcal{Q}_2|+|\mathcal{Q}_3|$. Putting the values of $|\Gamma_1|,|\Gamma_2|$, and $|\Gamma_3|$ and then solving we get $|\mathcal{Q}_1|=2(q_1+q_2-q_3)+1$, $|\mathcal{Q}_2|=2(q_1-q_2+q_3)+1$, and $|\mathcal{Q}_3|=2(-q_1+q_2+q_3)-1$.
		\end{proof}
		
		The theorem below gives complete characterization of unimodular graphs in $\B(\theta)$.
		
		\begin{thm}\label{{thm2}}
				Let $B\in\B(\theta)$ with cycles $\Gamma_1$, $\Gamma_2$, and $\Gamma_3$. Then $B$ is unimodular if and only if it satisfies one of the following conditions:
				\begin{enumerate}
					\item $B$ has a unique perfect matching.
					
					\item $B$ has at least two independent cycles, and exactly one of the cycles $\Gamma_1$, $\Gamma_2$, and $\Gamma_3$ has length congruent to $0\pmod4$, while the other two have lengths congruent to $2\pmod4$.
				\end{enumerate}
		\end{thm}
		\begin{proof}
			Follows from Lemma \ref{nonperf}, Lemma \ref{tuniquni}, Lemma \ref{notuni} and Corollary \ref{cor5}.
		\end{proof}
	
	\section{Determinant of graphs in $\B$}
	
		\begin{lem}\label{lem4.1}
			If $B\in\B$ has exactly one perfect matching, then $\det A(B)\in\{\pm1,\pm3,\pm5\}$.
		\end{lem}
		\begin{proof}
			Follows from Lemma \ref{iuniqueperuni}, Lemma \ref{iuniq}, and Lemma \ref{tuniquni}.
		\end{proof}
	
		\begin{lem}\label{lem4.2}
			If $B\in\B$ has more than one perfect matching, then $\det A(B)\in \{0,\pm1,\pm4,\pm9,\pm16\}$.
		\end{lem}
		\begin{proof}
			Follows from Lemma \ref{cor2}, Lemma \ref{notuni}, and Lemma \ref{detAB}.
		\end{proof}

		To complete the analysis, we now compute the determinant of $B \in \B$ when it has no perfect matching.

		\begin{lem}\label{lem4.3}
			Let $B\in\B$. If $B$ does not possess a perfect matching, then
			\[
			\det A(B)\in\{0,\pm2,\pm4,\pm8\}.
			\]
		\end{lem}
		
		\begin{proof}

		Since $B$ has no perfect matching, every spanning elementary subgraph of $B$ contains a cyclic component. Hence, by Remark~\ref{det2}, $\det A(B)$ is even.
		
		Suppose first that $B$ has exactly one independent cycle, then $B$ has exactly one spanning elementary subgraph. Thus Remark~\ref{det2} yields
		\[
		\det A(B)\in\{0,\pm2\}.
		\]
		
		Next, suppose that $B$ has an independent cycle $\Gamma$ such that $B-\Gamma$ contains another independent cycle $\Gamma'$. Then $\Gamma$ and $\Gamma'$ are vertex-disjoint, and $\Gamma'$ is even since $B-\Gamma$ has a perfect matching. By Lemma~\ref{mperf}, $B-\Gamma$ has exactly two perfect matchings, say $\M_1$ and $\M_2$. Consequently, $\Gamma\cup\M_1$ and $\Gamma\cup\M_2$ are two distinct spanning elementary subgraphs of $B$. Moreover, since $\Gamma$ and $\Gamma'$ are vertex-disjoint, the graph consisting of $\Gamma$, $\Gamma'$, and a perfect matching of $B-(\Gamma\cup\Gamma')$ forms another spanning elementary subgraph of $B$. Hence, by Remark~\ref{det2},
		\[
		\det A(B)\in\{0,\pm8\}.
		\]
		
		Now suppose that $B$ contains two independent cycles. As a result, $B$ has exactly two spanning elementary subgraphs each containing one of the independent cycle. Now, by Remark~\ref{det2},
		\[
		\det A(B)\in\{0,\pm4\}.
		\]
		
		Finally, suppose that $B-(\Gamma\cup\Gamma')$ has a perfect matching $\M$, where $\Gamma$ and $\Gamma'$ are vertex-disjoint but neither is independent in $B$. Then $\M\cup\Gamma\cup\Gamma'$ is the unique spanning elementary subgraph of $B$. Therefore, by Remark~\ref{det2},
		\[
		\det A(B)
		=
		4(-1)^{\frac{n+|\Gamma|+|\Gamma'|}{2}}
		\in\{\pm4\}.
		\]
		
		Since every independent cycle in a graph without a perfect matching is odd (Remark~\ref{rem1}), and a bicyclic graph contains at most two odd cycles, the above cases exhaust all possible configurations. Therefore,
		\[
		\det A(B)\in\{0,\pm2,\pm4,\pm8\},
		\]
		as required.
	\end{proof}

	In the following theorem we list all the possible values of $\det A(B)$ for $B\in\B$.
		
		\begin{thm}
			Let $B\in\B$. Then $\det A(B)\in\{0,\pm 1,\pm 2,\pm3,\pm 4,\pm5,\pm8,\pm 9,\pm 16\}$.
		\end{thm}
		\begin{proof}
			Follows from Lemma~\ref{lem4.1}, Lemma~\ref{lem4.2}, and Lemma~\ref{lem4.3}
		\end{proof}
	
%	\section{Conclusion}
%	
%	This study establishes conditions that are necessary and sufficient for the unimodularity of bicyclic graphs and determines all possible determinant values for such graphs. The findings provide a basis for investigating the invertibility and eigenvalue properties of bicyclic graphs and offer a framework, via the concept of independent cycles, for extending the characterization to tricyclic unimodular graphs.
	
	\section*{Acknowledgements}
		
		The author expresses gratitude to his Ph.D. supervisor, Dr. Debajit Kalita, for his guidance throughout the research. The author also acknowledges the financial assistance provided by the UGC, Government of India, under the UGC SRF scheme, bearing NTA Reference No.: 191620004617.
	
%	\bibliography{Uni_reference}
%	\bibliographystyle{plain}

\end{document}